\documentclass[a4paper,11pt]{amsart}
\usepackage{enumitem}
\usepackage{comment, amsmath, amssymb, amsgen, amsthm, amscd, xspace, color, epsfig, float, epic}
\usepackage{genyoungtabtikz}
\usepackage{tikz}
\usetikzlibrary{arrows.meta}
\usepackage{tikz-qtree}

\usepackage{pifont}
\usepackage{multirow}
\usepackage{threeparttable}
\usepackage[colorlinks,linkcolor=blue,urlcolor=cyan,citecolor=blue,pagebackref]{hyperref}
\makeatletter

\newcommand{\Rmnum}[1]{\expandafter\@slowromancap\romannumeral #1@}

\newcommand{\hobox}[3]{\draw (0+#1,0-#2) rectangle (1+#1,-1-#2)++(-0.5,+0.5) node {$ #3$};}

\newcommand{\domscale}{0.5}
\newcommand{\pf}{\begin{proof}}
	\newcommand{\epf}{\end{proof}}
\newcommand{\eq}{\begin{equation}}
	\newcommand{\eeq}{\end{equation}}
\newcommand{\eqn}{\begin{equation*}}
	\newcommand{\eeqn}{\end{equation*}}

\newcommand{\frg}{\mathfrak{g}}
\newcommand{\frh}{\mathfrak{h}}

\newcommand{\frl}{\mathfrak{l}}

\newcommand{\frn}{\mathfrak{n}}

\newcommand{\frq}{\mathfrak{q}}

\newcommand{\fru}{\mathfrak{u}}

\newtheorem{Thm}[equation]{Theorem}
\newtheorem{Cor}[equation]{Corollary}
\newtheorem{prop}[equation]{Proposition}
\newtheorem{lem}[equation]{Lemma}

\newtheorem{Rem}[equation]{Remark}

\theoremstyle{definition}
\newtheorem{definition}[equation]{Definition}

\newtheorem{example}[equation]{Example}

\numberwithin{equation}{section}
\usepackage{xcolor} 

\begin{document}

\title[Smooth right cells]{A characterization of Kazhdan--Lusztig right cells containing smooth elements }

\author{Zhanqiang Bai}
\address[Bai]{School of Mathematical Sciences, Soochow University, Suzhou 215006, P. R. China}
\email{zqbai@suda.edu.cn}

\author{Zheng-an Chen}
\address[Chen]{School of Mathematical Sciences,
	Shanghai Jiaotong University, Shanghai 200240,
	 P. R. China}
\email{zhengan$\_$chen@sjtu.edu.cn}

\bigskip

\begin{abstract}
Let $\mathfrak{g}$ be the Lie algebra  $\mathfrak{sl}(n,\mathbb{C})$. Its Weyl group is the symmetric group $S_n$.
In this paper, we want to describe some Kazhdan--Lusztig right cells  containing  {smooth elements which  parameterize the smooth Schubert varieties. These elements  are}  closely related to the study of associated varieties of highest weight modules of $\mathfrak{sl}(n,\mathbb{C})$. Firstly, we give a complete classification of the KL right cells containing only smooth elements. Then we give a sufficient condition for a KL right cell {to contain}  only non-smooth elements by using invariant subsequences and  a sufficient condition for a KL right cell  {to contain} some smooth elements. Finally, we give an efficient algorithm to find out all  the  smooth elements in a given KL right cell.

\end{abstract}

\maketitle

	\tableofcontents
\section{ Introduction}

In their famous paper \cite{KL}, Kazhdan and Lusztig introduced the concepts of right, left and two-sided cells in order to study representations of the Hecke algebras associated to a Coxeter group $W$. Now these concepts are studied by many people from representation theory and combinatorics.

Let $G=SL(n,\mathbb{C})$ be the special linear group.
Let $\mathfrak{g}$ be its simple complex Lie algebra and $\mathfrak{h}$ be a Cartan subalgebra.
Let $\Phi^+\subset\Phi$ be the set of positive roots determined by a Borel subalgebra $\mathfrak{b}$ of $\mathfrak{g}$. Denote by $\Pi$ the set of simple roots in $\Phi^+$. We fix a Borel subgroup $B\subset G$ corresponding to $\mathfrak{b}$. We have  a triangular decomposition $\mathfrak{g}=\mathfrak{n}\oplus \mathfrak{h} \oplus \mathfrak{n}^-$.
The Weyl group $W$ of $\mathfrak{g}$ is  $S_n$.

For $\lambda\in\mathfrak{h}^*$, the {\it  Verma module} $M(\lambda)$ is defined by
\begin{equation*}
M(\lambda)=U(\mathfrak{g})\otimes_{U(\mathfrak{b})}\mathbb{C}_{\lambda-\rho},
\end{equation*}
where $\mathbb{C}_{\lambda-\rho}$ is a one-dimensional $\mathfrak{b}$-module with weight $\lambda-\rho$ and $\rho$ is  half the sum of positive roots. Denote by $L(\lambda)$ the simple quotient of $M(\lambda)$.


  We use $L_w$ to denote the simple highest weight $\mathfrak{g}$-module of highest weight $-w\rho-\rho$ with  $w\in W$.
 Joseph \cite{Jo84} proved that the associated variety $V(L_w)$ is a union of orbital varieties defined as follows. Let $\mathcal{O}\subseteq \mathfrak{g}$ be a nilpotent $G$-orbit. The irreducible components of $\overline{\mathcal{O}}\cap \mathfrak{n}$ are called {\it orbital varieties} of  $\mathcal{O}$. They all take the form $\mathcal{V}(w)=\overline{B(\mathfrak{n}\cap w\mathfrak{n})}$ for some $w\in W$. Melnikov \cite{FM,Mel04,Mel04-2,Mel06} did a lot of work for the properties of orbital varieties of type $A$. The associated variety  $V(L_w)$ is  called {\it irreducible} if and only if it contains only one orbital variety. 
For a long time, people conjectured that the associated variety of any highest weight module $L_w$ is irreducible in the case of type $A$ (see \cite{BoB3} and \cite{Mel}). However,
Williamson \cite{Wi} showed that there exist counter-examples in 2014. Thus the structure of $V(L_w)$ or $V(L(\lambda)) $ is still mysterious in type $A$.

We refer to \cite{KL} or \S \ref{sec:cell} for the definition of Kazhdan--Lusztig right (resp. left and two-sided)  cell equivalence relation and use $\stackrel{R}{\sim}$ (resp. $\stackrel{L}{\sim}$ and $\stackrel{LR}{\sim}$ )  to denote the right  (resp. left and two-sided) cell equivalence relation.

Let $X_w=\overline{BwB/B}$ be the Schubert variety indexed by $w$ in the flag manifold  $G/B$. Lakshmibai and Sandhya \cite{LS} determined the smoothness of Schubert varieties for type $A$ by using pattern avoidance. It is known that for type $A$, a Schubert variety is  smooth if and only if  certain Kazhdan–-Lusztig polynomials are trivial, see for example \cite{CK}.

 From Sagan \cite{Sagan} or Bai--Xie \cite[Lemma 4.1]{BX}, we know that there is a bijection between the KL right cells in the symmetric group $S_n$ and the Young tableaux through the famous Robinson--Schensted insertion  algorithm.
We use $P(w)$ to denote the corresponding Young tableau for any $w\in S_n$.

From \cite{Jo84}, we know that  the associated variety $V (L_w)$ is constant on each KL right cell.
It is also known that if the associated variety $V (L_w)$ is reducible, the Schubert variety $X_w$ will be singular, see for example \cite[Corollary 4.3.2]{BoB3}. So there is a relationship between the reducibility  of associated varieties and non-smoothness (or smoothness)
of Schubert varieties.

In this paper, we consider the following problem:
{\it For which Kazhdan--Lusztig right cell $\mathcal{C}_R$, the  Schubert variety $X_w$ is smooth for every $w\in \mathcal{C}_R$}? Equivalently, the  Kazhdan--Lusztig polynomial $P_{e,w} (q) =1$ for all $w$ in these KL right cells $\mathcal{C}_R$.

For other KL right cells, we will give an algorithm to determine that it contains smooth elements or not. We have incorporated  this in a
simple  program. The program takes some $w$ as the input and returns the set of smooth elements.  It is available at
\begin{center}
\text{{https://github.com/zhengan-chen/Young$\_$tableaux}}.
    \end{center}
Using our program, we can easily find many elements in $S_n$ for which the corresponding associated variety of $L_w$ is irreducible. See Corollary \ref{findsmooth}.



This paper is organized as follows. In \S \ref{pre}, we  prepare some necessary preliminaries on associated varieties, pattern avoidance and KL right cells. In \S \ref{smoothcell},
we will give a complete classification of KL right cells containing only smooth elements, see Theorem \ref{main}. Then
in \S \ref{nonsmoothcell}, we  describe some special KL right cells containing only non-smooth elements. In \S \ref{twocolumn},  we  give a characterization for some special KL right cells (containing smooth elements) corresponding
to some Young tableaux with two columns. In \S \ref{algorithm}, we give an algorithm to find out all  the  smooth elements in a given KL right cell.

\section{Preliminaries}\label{pre}
In this section, we give some brief preliminaries on associated varieties of highest weight modules, pattern avoidance and the Robinson--Schensted insertion algorithm.

\subsection{Associated variety}
Let $\mathfrak{g}$ be a simple complex Lie algebra. Let $M$ be a  {finitely} generated $U(\mathfrak{g})$-module. Fix a finite dimensional generating subspace $M_0$ of $M$. Let $U_{n}(\mathfrak{g})$ be the standard filtration of $U(\mathfrak{g})$. Set $M_n=U_n(\mathfrak{g})\cdot M_0$ and
\(
\text{gr} (M)=\bigoplus\limits_{n=0}^{\infty} \text{gr}_n M,
\)
where $\text{gr}_n M=M_n/{M_{n-1}}$. Thus $\text{gr}(M)$ is a graded module of $\text{gr}(U(\mathfrak{g}))\simeq S(\mathfrak{g})$.



\begin{definition}
	The  \textit{associated variety} of $M$ is defined by
	\begin{equation*}
	V(M):=\{X\in \mathfrak{g}^* \mid f(X)=0 \text{ for all~} f\in \operatorname{Ann}_{S(\mathfrak{g})}(\operatorname{gr} M)\}.
	\end{equation*}
\end{definition}

The above  definition is independent of the choice of $M_0$ (e.g., \cite{NOT}).

\begin{definition} Let $\mathfrak{g}$ be a finite-dimensional semisimple Lie algebra. Let $I$ be a two-sided ideal in $U(\mathfrak{g})$. Then $\text{gr}(U(\mathfrak{g})/I)\simeq S(\mathfrak{g})/\text{gr}I$ is a graded $S(\mathfrak{g})$-module. Its annihilator is $\text{gr}I$. We define its associated variety by
	$$V(I):=V(U(\mathfrak{g})/I)=\{X\in \mathfrak{g}^* \mid p(X)=0\ \mbox{for all $p\in {\text{gr}}I$}\}.
	$$
\end{definition}

 We have the following proposition.

\begin{prop}[ \cite{Jo85}]
	Let $\mathfrak{g}$ be a reductive Lie algebra and $I$ be a primitive ideal in $U(\mathfrak{g})$.Then $V(I)$ is the closure of a single nilpotent coadjoint orbit $\mathcal{O}_I$ in $\mathfrak{g}^*$. In particular, for a highest weight module $L(\lambda)$, we have $V(\mathrm{Ann} (L(\lambda)))=\overline{\mathcal{O}}_{\mathrm{Ann}(L(\lambda))}$.
\end{prop}

Let $G$ be a connected semisimple finite dimensional complex algebraic group with Lie algebra $\mathfrak{g}$. We fix some triangular decomposition $\mathfrak{g}=\mathfrak{n }\oplus \mathfrak{h} \oplus \mathfrak{n}^{-}$. Let $\mathcal{O}$ be a nilpotent $G$ orbit. The   irreducible components of $\overline{\mathcal{O}}\cap \mathfrak{n}$  are called {\it orbital varieties} associated to $\mathcal{O}$.  {After Steinberg \cite{Ste76}}, an orbital variety has the following form
$$\mathcal{V}(w)=\overline{B(\mathfrak{n}\cap{w(\mathfrak{n})} )}$$ for some  $w$ in the Weyl group $ W$ of $\mathfrak{g}$, where $B$ is the Borel subgroup of $G$ corresponding the Borel subalgebra $\mathfrak{b}$ of $\mathfrak{g}$.

We have the following  propositions.

\begin{prop}[\cite{Jo84}]
	Let $L(\lambda)$ be a highest weight module of a simple Lie algebra $\mathfrak{g}$ with highest weight $\lambda-\rho$. Then its associated variety $V(L(\lambda))$ equals the union of some orbital varieties associated with the nilpotent coadjoint orbit $\mathcal{O}_{Ann(L(\lambda))}$ in $\mathfrak{g}^*$.
	
\end{prop}

The associated variety $V(L(\lambda))$ is  {irreducible}  {if and only if it coincides with a single orbital variety}.

\begin{prop}[{\cite[Corollary 4.3.2]{BoB3}}]\label{smoothequal}
	If the Schubert variety $X_w=\overline{BwB/B}$ is smooth, we will have
$ V(L_w)=\mathcal{V}(w)$.	

\end{prop}

\subsection{Pattern avoidance}

By the definition, an element $w\in S_n$ is a permutation of the set $\{1,2,...,n\}$. In general, we use $w=(w_1,...,w_n)$ to denote this permutation, where $w_i=w(i)$.

Firstly we have the following definition.

\begin{definition}
	The element $w = (w_1, ..., w_n)\in S_n$ {\it contains the pattern} $3412$
	 (resp. $4231$) if there exist integers $1\leq i < j < k < l\leq n$ such that
	$w_k < w_l < w_i < w_j $ (resp. $ w_l < w_j < w_k < w_i$).
	 If there is no such integers, we say $w$ {\it avoids the pattern } $3412$ and $4231$.
\end{definition}

%
%
%
%

We have the following criterion for smoothness of Schubert varieties.

\begin{prop}[\cite{LS}]
	For $\mathfrak{g}=\mathfrak{sl}(n, \mathbb{C})$ and $ W=S_n$,  the Schubert variety $X_w=\overline{BwB/B}$ is smooth if and only if
	$ w$ avoids the two patterns $3412$ and $4231$.
\end{prop}

In general, $w$ is called a smooth element when $X_w$ is smooth.

\subsection{Hecke algebra and cells }\label{sec:cell}
Recall that the Weyl group $ W  $ is a Coxeter group generated by $ S=\{s_\alpha\mid\alpha\in\Delta \} $. Let $\ell(-)$ be the length function on $W$.
 Given an indeterminate $v$, the Hecke algebra $ \mathcal{H} $ over $ \mathcal{A} :=\mathbb{Z}[q,q^{-1}]$ is generated by $ T_w $, $ w\in W $ with relations \[
T_wT_{w'}=T_{ww'} \text{ if }\ell(ww')=\ell(w)+\ell(w'),
\]
\[
\text{and }(T_s+q^{-1})(T_s-q)=0 \text{ for any }s\in S.
\]
The unique elements $ C_w $ such that
\[
\overline{C_w}=C_w,\qquad C_w\equiv T_w \mod{\mathcal{H}_{<0}}
\]
are known as the \textit{Kazhdan--Lusztig} (KL) \textit{basis} of $ \mathcal{H} $, where $ \bar{\,} :\mathcal{H}\rightarrow\mathcal{H}$ is the bar involution such that $ \bar{q}=q^{-1} $, $ \overline{T_w} =T_{w^{-1}}^{-1}$, and $ \mathcal{H}_{<0}=\bigoplus_{w\in W}\mathcal{A}_{<0}T_w $ with $ \mathcal{A}_{<0}=q^{-1}\mathbb{Z}[q^{-1}] $.

If $ C_y $ occurs in the expansion of $ hC_w $ (resp. $C_wh$) with respect to the KL-basis for some $ h\in\mathcal{H} $, then we write $ y\leftarrow_L w $ (resp. $ y\leftarrow_R w $). Extend $ \leftarrow_L $ (resp. $ \leftarrow_R $) to a preorder $ \leq_L $ (resp. $\leq _R$) on $ W $. For $x, w\in W$, write $x \leq_{LR} w$ if there exists $x=w_1, \cdots, w_n=w$ such that for every $1\leq i<n$ we have either $w_i\leq_L w_{i+1}$ or $w_i\leq_R w_{i+1}$. Let $\stackrel{L}{\sim}$, $\stackrel{R}{\sim}$, $\stackrel{LR}{\sim}$ be the equivalence relations associated with $\stackrel{L}{\sim}$, $\stackrel{R}{\sim}$, $\stackrel{LR}{\sim}$ (for example, $x\stackrel{L}{\sim}w$ if and only if $x\leq_L w$ and $w\leq_Lx$). The equivalence classes on $W$ for $\stackrel{L}{\sim}$, $\stackrel{R}{\sim}$, $\stackrel{LR}{\sim}$ are called \textit{left cells}, \textit{right cells} and \textit{two-sided cells} respectively.

\begin{prop}[{\cite[II. 9.8]{Spa}}]\label{geometriccell}
    Suppose  $x,y\in S_n$. Then we have $\mathcal{V}(x)=\mathcal{V}(y)$ if and only if $x\stackrel{R}{\sim} y$.
\end{prop}

\begin{prop}[{\cite[Lemma 6.6]{Jo84}}; {\cite[Corollary 6.3]{BoB3}}]\label{constant}
$V(L_w)$ is constant on each KL right cell.	
	
\end{prop}

\subsection {Robinson--Schensted insertion algorithm}
In this subsection, we recall the famous Robinson--Schensted insertion  algorithm. Some details can be found in \cite{Ar} and \cite{Sagan}.

\begin{definition}[Robinson--Schensted insertion algorithm]
For an element  $ w\in\ S_n $, we write  $w=(w_1,...,w_n)$. We associate to $w $ a  Young tableau  $ P(w) $ as follows. Let $ P_0 $ be an empty Young tableau. Assume that we have constructed Young tableau $ P_k $ associated to $ (w_1,\cdots,w_k) $, $ 0\leq k<n $. Then $ P_{k+1} $ is obtained by adding $ w_{k+1} $ to $ P_k $ as follows. Firstly we add $ w_{k+1} $ to the first row of $ P_k $ by replacing the leftmost entry $ x $ in the first row which is \textit{strictly} bigger than $ w_{k+1} $.  (If there is no such an entry $ x $, we just add a box with entry $w_{k+1}  $ to the right side of the first row, and end this process). Then add $ x $ to the next row as the same way of adding $w_{k+1} $ to the first row.  Finally we put $P(w)=P_n$.
Let $Q(w)$ be the recording tableau such that $1,2,...,n$ are placed in the $Q$'s so that shape of $P_k$ equal to the shape of $Q_k$ for all $1\leq k\leq n$. Thus $Q(w)=Q_n$ and $\mathrm{sh}(P(w))=\mathrm{sh}(Q(w))$.
\end{definition}

We use $p(w)=[p_1,...,p_k]$ to denote the shape of $P(w)$, where $p_i$ is the number of boxes in the $i$-th row of  $P(w)$. So  $[p_1,...,p_k]$
is a partition of $n$, denoted by $p(w)\vdash n$.

\begin{prop}[{\cite[Theorem 3.1.1 $\&$ Theorem 3.6.6]{Sagan}}]
\label{prop::2.11}
The map $w\rightarrow (P(w),Q(w))$ is a bijection between elements of $S_n$ and pairs of standard tableaux of the same shape
$p(w)\vdash n$. We also have $P(w^{-1})=Q(w)$ and $Q(w^{-1})=P(w)$ for any $w\in S_n$.
\end{prop}

\begin{prop}[ \cite{Ar} or \cite{KL}]
\label{prop::2.12}
	For $\mathfrak{g}=\mathfrak{sl}(n, \mathbb{C})$ and $ W=S_n$,  two elements $x$ and $ y$ in $S_n$ are in the same KL right cell if and only if $P(x)=P(y)$.
\end{prop}

\begin{definition}
Suppose $x<y<z$. Then $w, \pi \in S_n$ differ by a Knuth relation of the first kind, if $$w=(w_1,...,y,x,z,...,x_n)  \text{~and~} \pi=(x_1,...,y,z,x,...,x_n) \text{~or vice versa}. $$
They differ by a Knuth relation of the second kind, if $$w=(w_1,...,x,z,y,...,x_n) \text{~and~} \pi=(x_1,...,z,x,y,...,x_n) \text{~or vice versa}. $$
The two elements are \textit{Knuth equivalent}, written $w \stackrel{K}{\sim}\pi$, if there is  a sequence of   elements in $S_n$
such that $w=\pi_1,...,\pi_N=\pi$ such that $\pi_i$ and $\pi_{i+1}$ differ  by a Knuth relation of the first kind or second kind for all $1\leq i\leq N-1$.
\end{definition}

\begin{prop}[ \cite{K}]
	For   two elements $x$ and $ y$ in $S_n$, we have $x \stackrel{K}{\sim}y$ if and only if $P(x)=P(y)$.
\end{prop}

\begin{definition}
If $P$ is a Young tableau, then the \textit{row word} of $P$ is the permutation
$$\pi_P=R_NR_{N-1}\cdots R_1,$$
where $R_i$ is the $i$-th row of $P$.

The \textit{column word} of $P$ is the permutation
$$\pi^c_P=C^t_1C^t_{2}\cdots C^t_k,$$
where $C_i$ is the $i$-th column of $P$ and $C^t_i$ is the transpose of $C_i$ and is starting from the bottom of $C_i$.
\end{definition}

It is important to  {emphasize} that   {$P(\pi_P) = P(\pi^c_P) = P$}.

\begin{example}Suppose
	\[
	P(w)={\begin{tikzpicture}[scale=\domscale+0.1,baseline=-18pt]
	\hobox{0}{0}{1}
	\hobox{1}{0}{3}
	\hobox{0}{1}{2}
	\hobox{1}{1}{4}
	\hobox{2}{0}{5}
	\end{tikzpicture},}
	\]
then the corresponding row word is $\pi_P=(2,4,1,3,5)$ and column word is $\pi^c_P=(2,1,4,3,5)$.
\end{example}

The following propositition is very useful in the proof our main results.
\begin{prop}[{\cite[Theorem 3.5.3]{Sagan}}]\label{rowcolumn}
	Consider $\pi \in S_n $. The length of the longest increasing subsequence of $\pi$ is the length of the first row of $P(\pi)$.
	The length of the longest decreasing subsequence of $\pi$ is the length of the first column of $P(\pi)$.
\end{prop}



\section{Smooth KL right cells}\label{smoothcell}
In this section, we will determine the KL right cells where all elements are smooth. Our result is based on a simple observation.
\begin{lem}\label{ii1}
	Let $\pi \in S_n$, for every $i$, the relative position of $i$ and $i+1$ remains unchanged under the action of the Knuth relation.
\end{lem}
\begin{proof}
	Based on the definition of the Knuth relation, if we want to change the position of two elements, we need the help of the middle number. But for {$i$} and  {$i+1$}, there is no such middle number. Thus, the lemma holds.
\end{proof}


 Before we give the proof of our result,  we discuss some special cases.

 \begin{prop}\label{two}
 	Let P be a Young tableau. The shape of P is $[2,1,\dots,1]=[2,1^{n-2}]$. Let $\pi_P=(n,\dots,k+1,k-1,\dots,2,1,k)$ be  the row word of $P$. Then the right cell corresponding to the permutation $\pi_P$ will be one of the   following two types:
 	\begin{enumerate}
 		\item All elements in the right cell corresponding to $\pi_P=(n,\dots,3,1,2)$ or $\pi_P=(n-1,\dots,2,1,n)$ will  be smooth;
 		
 		\item The right cell contains smooth and non-smooth elements for $2<k<n$.
 	\end{enumerate}
 \end{prop}

 \begin{proof} We use $[n,m]$ to denote the set $\{n,n+1,...,m\}$ for $n<m$. We use $``k"$ to denote the $k$-th largest element in the pattern $4231$ and the pattern $3412$ for a given permutation $\pi \in S_n$.
 	
 	For case one, we assert that the pattern $3412$ cannot appear.  {To justify this, consider the permutation $\pi_P = (n,\dots,3,1,2)$.
    Here, the element ``3" must be chosen from $[3,n]$ (since 1 and 2 are fixed at the end), but Lemma~\ref{ii1} enforces that any larger element ``4" must precede ``3" in $\pi_P$.
    Consequently, the relative order of ``4" and ``3" makes the $3412$ pattern impossible. }
    For the pattern $4231$, if we consider $\pi_P=(n,\dots,3,1,2)$, you will find that we can only take $``2"$ from $[2, n]$. But Lemma \ref{ii1} implies possible $``3"$ is before $``2"$, which implies that the pattern $4231$ will not appear. If we consider $\pi_P=(n-1,\dots,2,1,n)$, the discussion is similar.
     {Here we give the full list of elements in the right cell of $(n,\dots,3,1,2)$ and $(n-1,\dots,2,1,n)$ in the following Table \ref{el-cell}.}
    \begin{table}[H]
    \centering
    \footnotesize{\renewcommand{\arraystretch}{1.4} 
    \begin{tabular}{|c|l|}
    \hline
    Permutation & Elements in the right cell \\
    \hline
    $(n, \dots, 3, 1, 2)$
    & $(n, \dots, 3, 1, 2),(n, \dots, 4, 1, 3,  2), \dots, (1, n, n-1, \dots, 3, 2)$ \\
    \hline
    $(n-1, \dots, 2, 1, n)$
    & $(n-1, \dots, 2, 1, n), (n-1, \dots, 2, n, 1),\dots, (n-1, n, n-2, \dots, 2, 1)$ \\
    \hline
    \end{tabular}}
    \caption{List of elements in the right cell for the first case.}\label{el-cell}
    \end{table}

 	For the case $(2)$, we find that there are no pattern $3412$ and $4231$ in the initial permutation. But after some Knuth relation, the pattern $4231$ will appear:
 	$$
 	\pi_P=(n-1,\dots,k+1,k-1,\dots,2,1,k)\stackrel{K}{\sim}(n-1,\dots,k+1,k-1,\dots,2,k,1)
 	$$
 since $2<k<n$.	
 	Now, $(k+1,k-1,k,1)$ will satisfy the pattern $4231$.
 \end{proof}

\begin{prop}
	Let $P$ be a Young tableau. The shape of $P$ is $[n-1,1]$. Then the elements in the right cells are all smooth.
\end{prop}

\begin{proof}
	Let $\pi_P=(k,1,2,\dots,k-1,k+1,\dots,n)$ for some $2\leq k\leq n$.
	
	For the pattern $4231$, note that it has a subsequence $(4,2,1)$. So by Proposition \ref{rowcolumn}, the length of the first column of $P(\pi_P)$ is larger than $2$, which is a contradiction.
	
	For the pattern $3412$, if $``3"$ is taken from $[1, k-1]$, then we can't find $``12"$. If $``3"$ is taken from $[k, n]$, we will take $``12"$ from $[1, k-1]$. But now the length of the longest increasing subsequence is less than before, which is a contradiction.
	
	Thus, the pattern $4231$ and $3412$ can not appear in any element of any KL right cell corresponding to some $\pi_P=(k,1,2,\dots,k-1,k+1,\dots,n)$.
\end{proof}

\begin{Rem}
	From the above two propositions, we can see that we cannot simply use symmetry to make inferences. The nature of $P$ and $P^t$ may be completely different.
\end{Rem}
 Now we give the theorem:

\begin{Thm}\label{main}
Let $n\in \mathbb{Z}_{>0}$.
Suppose $P$ is a standard Young tableau
with $n$ boxes. Then its corresponding KL right cell  consists entirely of smooth elements if and
only if $P$ is one of the followings:

\begin{enumerate}
	\item When the shape of $P$ is $[n]$ or $[1,\dots,1]=[1^n]$, we will have:
	\[
	P={
	\begin{tikzpicture}[scale=\domscale+0.1,baseline=-13pt]
	\hobox{0}{0}{1}
	\hobox{1}{0}{2}
	\hobox{2}{0}{\dots}
	\hobox{3}{0}{n}
	\end{tikzpicture}}
	\] or
	\[
	P={	
	\begin{tikzpicture}[scale=\domscale+0.1,baseline=-36pt]
	\hobox{0}{0}{1}
	\hobox{0}{1}{2}
	\hobox{0}{2}{\vdots}
	\hobox{0}{3}{n}
	\end{tikzpicture}};
	\]
	\item When the shape of $P$ is $[2,1,\dots,1]=[2,1^{n-1}]$ or $[n-1,1]$, we will have:
	\[
	P=	{
	\begin{tikzpicture}[scale=\domscale+0.1,baseline=-38pt]
	\hobox{0}{0}{1}
	\hobox{1}{0}{2}
	\hobox{0}{1}{3}
	\hobox{0}{2}{\vdots}
	\hobox{0}{3}{n}
	\end{tikzpicture}}
	\]or
	 \[
	P=	\scriptsize{
	\begin{tikzpicture}[scale=\domscale+0.3,baseline=-47pt]
	\hobox{0}{0}{1}
	\hobox{1}{0}{n}
	\hobox{0}{1}{2}
	\hobox{0}{2}{\vdots}
	\hobox{0}{3}{n-1}
	\end{tikzpicture}}
	\]	
	or
	\[
	P=\scriptsize{	
	\begin{tikzpicture}[scale=\domscale+0.3,baseline=-25pt]
	\hobox{0}{0}{1}
	\hobox{0}{1}{k}
	\hobox{1}{0}{2}
	\hobox{2}{0}{\dots}
	\hobox{3}{0}{k-1}
	\hobox{4}{0}{k+1}
	\hobox{5}{0}{\dots}
	\hobox{6}{0}{n}
	\end{tikzpicture}}
	\]
     {with  $k\in \{ 2, \dots, n\}$;}
	
	\item 	When the shape of $P$ is $[k,1,\dots,1]=[k,1^{n-k}]$ for some $2<k<n-1$ and the length of the first column is larger than $2$, we will have:
	\[
	P=\scriptsize{	
	\begin{tikzpicture}[scale=\domscale+0.3,baseline=-47pt]
	\hobox{0}{0}{1}
	\hobox{0}{1}{2}
	\hobox{0}{2}{\vdots}
	\hobox{0}{3}{i}
	\hobox{1}{0}{i+1}
	\hobox{2}{0}{\dots}
	\hobox{3}{0}{n}
	\end{tikzpicture}}
	\]
	or	
	\[
	P=\scriptsize{	
	\begin{tikzpicture}[scale=\domscale+0.3,baseline=-47pt]
	\hobox{0}{0}{1}
	\hobox{1}{0}{2}
	\hobox{2}{0}{\dots}
	\hobox{3}{0}{i}
	\hobox{0}{1}{i+1}
	\hobox{0}{2}{\vdots}
	\hobox{0}{3}{n}
	\end{tikzpicture}}
	\]
	 {with $i \in \{ 3, \dots, n-2\}$}, or	
	\[
	P=\scriptsize{	
	\begin{tikzpicture}[scale=\domscale+0.3,baseline=-47pt]
	\hobox{0}{0}{1}
	\hobox{1}{0}{2}
	\hobox{2}{0}{\dots}
	\hobox{3}{0}{i}
	\hobox{4}{0}{l}
	\hobox{5}{0}{\dots}
	\hobox{6}{0}{n}
	\hobox{0}{1}{i+1}
	\hobox{0}{2}{\vdots}
	\hobox{0}{3}{l-1}
	\end{tikzpicture}}
	\]
     {with $l \in \{ 5, \dots, n\}$ and $2\leq i\leq l-3 $};
	
	\item When the shape of $P$ is $[n-2,2]$, we will have:
	\[
	P=\scriptsize{	
	\begin{tikzpicture}[scale=\domscale+0.3,baseline=-25pt]
	\hobox{0}{0}{1}
	\hobox{1}{0}{3}
	\hobox{2}{0}{\dots}
	\hobox{3}{0}{k-1}
	\hobox{4}{0}{k+1}
	\hobox{5}{0}{\dots}
	\hobox{6}{0}{n}
	\hobox{0}{1}{2}
	\hobox{1}{1}{k}
	\end{tikzpicture}}
	\]
     {with $k \in \{ 4, \dots, n\}$}.
\end{enumerate}	
\end{Thm}

\begin{proof} {We give the arguments by considering the lengths of rows and columns of the corresponding Young tableaux.}

$(\textbf{Step A})$. First, we consider a Young tableau with only one row or one column, which will be the Young tableau in case $(1)$.	It is obvious that the elements of right cells corresponding to $\pi_P=(1,2,\dots,n)$ and $\pi_P=(n,n-1,\dots,1)$ are smooth, since we can't find any Knuth relations in those permutations.

$(\textbf{Step B})$.
Now, we consider Young tableaux with only one row and one column whose lengths are both larger than one. Since we have discussed  case $(2)$ in Proposition \ref{two}, we only focus on the tableaux in case $(3)$, whose lengths of the first row and
first column are both larger than $2$.

$(\textbf{Step C})$. Suppose $P$ is a Young tableau in case $(3)$. Let  $\pi_P$ be  denoted as  {$(a_{n-k}, \dots, a_2, 1, b_2, \dots, b_k)$ in which $a_2 =2$ or $b_2 = 2$,  $a_{n-k}>a_{n-k-1}>\cdots >a_2$ and $b_{2}<b_3<\cdots <b_k$. In the following, we set $a = a_2$ and $b = b_2$. Note that only the triple $(a, 1, b)$ gives rise to a Knuth relation in $\pi_P$}.

\begin{enumerate}[leftmargin=18pt]
    \item[1)] If $a=2<b$, we will have $\pi_P\stackrel{K}{\sim}  {(a_{n-k}, \dots, a, b, 1, \dots, b_k)}$. If the leftmost element, denoted as $c = a_{n-k}$, of $\pi_P$ is larger than $b$, $   (c,a,b,1)$ will be the pattern $4231$. So we assume that $c<b$, then we have $$\pi_P=(i,i-1,\dots,2,1,i+1,\dots,n),$$
    where $b=i+1$ and $3\leq i\leq n-2$. We claim that the elements of the right cell will avoid the pattern $3412$ and $4231$.

First, we consider the pattern $3412$. If $``3"$ is taken from $[3, i]$. By Lemma \ref{ii1} , we note that $``4"$ can only be taken from $[i+1, n]$ and $``1", ``2"$ can only be taken from $[1, ``3"-1]$. However, by Lemma \ref{ii1}, $``2"$ must be before $``1"$. It is absurd. If $``3"$ is taken from $[i+1, n]$, the discussion is the same. When we consider the pattern $4231$, we find that $``2"$ must be taken from $[1,i]$ and $``3"$ must be taken from $[i+1,n]$. Using Lemma \ref{ii1}, we note that $``4"$ must be after $``3"$, which is a contradiction. So all elements of this right cell will avoid the pattern $3412$ and $4231$.
\item[2)]  {Now we assume that $a > b=2$. Recall that the lengths of the first row and first column are both larger than $2$. Now we denote the right element of $b$ as $d$ and write $\pi_P$ as $$(a_{n-k},\dots,a_3, a, 1, 2, d, b_4, \dots, b_k).$$}

\begin{enumerate}[leftmargin=18pt]
    \item[i)]If  $a>d$, we have the follows.
\begin{enumerate}[leftmargin=18pt]
    \item[a)] If $a$ is larger than the last element of $\pi_P$, then we have $$\pi_P=(n,\dots,i+1,1,2,\dots,i).$$  We claim that the elements of the right cell are all smooth. For the pattern $3412$, if we take $``3"$ from $[i+1,n]$, we can not find $``4"$. If we take $``3"$ from $[1,i]$, we can not find $``2"$. As for the pattern $4231$, we know that $``4"$ can only be taken from $[i+1,n]$. Then if $``2"$ is also taken from $[i+1,n]$, we can't find $``3"$. If $``2"$ is taken from $[1,i]$, we can't find $``1"$.
\item[b)]If there exists an element in the first row of $P$ larger than $a$,  {denoted the first one as $l$ and the leftmost element $c$ of $\pi_P$ will be larger than $l$}, we can write $$\pi_P=(c,\dots,a,1,2,3,\dots,a-1, {l},\dots).$$ Using the Knuth relations,  we have:
\begin{align*}
(c,\dots,a,1,2,3,\dots,a-1, {l},\dots)&\stackrel{K}{\sim}(c,\dots,1,a,2,3,\dots,a-1, {l},\dots) \\
&\stackrel{K}{\sim}(c,\dots,1,2,a,3,\dots,a-1, {l},\dots) \\
&\stackrel{K}{\sim}(\dots) \\
&\stackrel{K}{\sim}(c,\dots,1,2,3,\dots,a,a-1, {l},\dots )\\
&\stackrel{K}{\sim}(c,\dots,1,2,3,\dots,a, {l},a-1\dots ).
\end{align*}
We find that $(c,a, {l},a-1)$ has the pattern $4231$, which implies that the right cell has a non-smooth element.
\item[c)]If the leftmost element is smaller than  {$l$}, we can write $$\pi_P=(  {l-1},\dots,i+1,1,2,3,\dots,i,l,
\dots).$$ For the pattern $3412$, if $``3"$ is taken from $[1,i]$, we can't find $``1"$ and $``2"$. If $``3"$ is taken from $[i+1,l-1]$, $``4"$ must be taken from $[l,n]$, then $``1"$ and $``2"$ must be taken from $[1,i]$. But the length of the longest increasing subsequence will decrease, which is a  contradiction. If $``3"$ is taken from $[l,n]$, the contradiction is the same.
For the pattern $4231$, if we take $``4"$ from $[i+1,l-1]$, then if we take $``2"$ from $[i+1,l-1]$, $``3"$ will have no choice. If we take $``2"$ from $[1,i]$, $``1"$ will have no choice. If we take $``4"$ from $[1,i]$,  $``1"$ and $``2"$ will have no choices. If we take $``4"$ from $[l,n]$, we  find that $``3"$ will have no choice.
\end{enumerate}


\item[ii)] Now we assume that $a<d$. By using the Knuth relation, we can  change $\pi_P$ into $$(a_{n-k},\dots,a_3, 1, a, d, 2,  b_4, \dots, b_k).$$ If  $a_{n-k}>d$, then $(a_{n-k},a,d,2)$ will be the pattern $4231$. So we assume  $a_{n-k}<d$, then  {$$\pi_P =(l-1,\dots,3,1,2,l,\dots,n)$$ for some $l>3$. Note that this $\pi_P$ corresponds to a special case of the third tableau in case $(3)$, i.e., $i=2$.}  First we claim that there is no pattern like $3412$. Observing $\pi_P$, we find that the longest increasing subsequence is $(1,2,l,\dots,n)$. But if we have the pattern $3412$, we find that $``1",``2"$ must be $1,2$ and $``4"$ must be taken from $[l,n]$. But when the pattern exists, the length of the longest increasing subsequence of the  permutation $\pi_P$ is less than $n-l+3$, which is a contraction. As for the pattern $4231$, we can easily find a contradiction in the relative position between $``2"$ and $``3"$.
\end{enumerate}
\end{enumerate}

$(\textbf{Step D})$.  Now we consider Young tableaux with multiple rows and columns whose lengths are larger than one.  We find that if the given Young tableau $P$ has more than three rows whose lengths are larger than one, the pattern $4231$ will appear  {in $\pi_P$}. And if the length of the second row is larger than $2$, the pattern $3412$ will also appear  {in $\pi_P$}.  {Now
it will be assumed that the second row has length 2 and each row below
the second row (if there is any) has length 1.} The first element of the second row, denoted it by $b$, must be $2$. Otherwise, we will find $b,c,1,a $ will satisfy the pattern $3412$. Thus, the Young tableau corresponding to $\pi_P$ and avoiding the pattern $3412$ and $4231$ must be  like:
\[
P(w)=\begin{tikzpicture}[scale=\domscale+0.1,baseline=-30pt]
\hobox{0}{0}{1}
\hobox{1}{0}{a}
\hobox{0}{1}{2}
\hobox{1}{1}{c}
\hobox{1}{0}{}
\hobox{2}{0}{\dots}
\hobox{0}{2}{\vdots}
\end{tikzpicture}.
\]

If there exists some element larger than $c$ in the first column, $\pi_P$ will have the pattern $4231$. So we can assume the elements in the first column are smaller than $c$.  {Moreover, if there is an entry larger than $a$ in the first column, then the pattern $3412$ appears in $\pi_P$. As a result, we can assume the entries of the first column are $1, 2, \dots, a-1$}.
\begin{enumerate}[leftmargin=18pt]
    \item[1)]If $P$ has only two rows, the corresponding tableau will be  like:

\[
P=\begin{tikzpicture}[scale=\domscale+0.1,baseline=-20pt]
\hobox{0}{0}{1}
\hobox{1}{0}{3}
\hobox{0}{1}{2}
\hobox{1}{1}{k}
\hobox{2}{0}{\dots}
\end{tikzpicture},
\]
and $\pi_P=(2,k,1,3,\dots)$.  For the pattern $3412$, if we take $``4"$ from $[3,k-1]$, we can't find the correct $``2"$. If we take $``4"$ from $[k+1,n]$, denoted it as $l$, then $``2"$ must be taken from $[3,k-1]$, denoted it as $j$. Now the length of the longest increasing subsequence is actually less than before, which is  a contradiction.  For  the pattern $4231$, since the length of the first column is $2$,  there will be no longer decreasing subsequence. So it avoids the pattern $4231$.
\item[2)]The last case we need to consider is:
\[
P=\scriptsize{\begin{tikzpicture}[scale=\domscale+0.3,baseline=-48pt]
\hobox{0}{0}{1}
\hobox{0}{1}{2}
\hobox{0}{2}{\vdots}
\hobox{0}{3}{k}
\hobox{1}{0}{k+1}
\hobox{1}{1}{j}
\hobox{2}{0}{\dots}
\hobox{3}{0}{n}
\end{tikzpicture}}
\]

Now $\pi_P=(k,k-1,\dots,3,2,j,1,k+1,\dots,j-1,j+1,\dots,n)$. We can easily find that
\begin{align*}
\pi_P=&(k,k-1,\dots,3,2,j,1,k+1,\dots,j-1,j+1,\dots,n)\\
\stackrel{K}{\sim}&(k,k-1,\dots,3,j,2,1,k+1,\dots,j-1,j+1,\dots,n)\\
\stackrel{K}{\sim}&(k,k-1,\dots,3,j,2, k+1,1,\dots,j-1,j+1,\dots,n).
\end{align*}
So $(j,2,k+1,1)$ satisfies the pattern $4231$. Thus $P$ contains  some non-smooth element.
\end{enumerate}

\end{proof}

We call a KL right cell in the above theorem a \emph{smooth KL right cell}.
If we use $\mathcal{C}_R(P)$ to denote the corresponding KL right cell for a given Young tableau $P$, we will have the following.
\begin{Cor}A KL right cell $\mathcal{C}_R(P)$ is a smooth cell if and only if its column word $\pi_P^c$ is  one of the following elements:
\begin{itemize}
    \item $x_1=(1,2,...,n)$;
    \item $x_2=(n,n-1,...,1)$;
    \item $x_3=(n,n-1,...,3,1,2)$;
    \item $x_4=(n-1,n-2,...,1,n)$;
    \item $y_k=(k,1,2,...,k-1,k+1,...,n)$, $2\leq k\leq n$;
    \item $z_i=(i,i-1,...,1,i+1,i+2,...,n)$, $3\leq i\leq n-2$;
    \item $z^{t}_i=(n,n-1,...,i+1,1,2,3,...,i)$, $3\leq i\leq n-2$;
    \item $s_{ik}=(l-1,l-2,\dots,i+1,1,2,\dots,i,l,l+1,\dots,n)$, $2\leq i\leq l-3\leq n-3$;
    \item $t_k=(2,1,k,3,4,...,k-1,k+1,...,n)$, $4\leq k\leq n$.

\end{itemize}
\end{Cor}

\section{ {Examples of KL right cells with only non-smooth elements}}\label{nonsmoothcell}

By Lemma \ref{ii1}, we can find some special KL right cells with all elements being  non-smooth. In general,  we prove that the permutations corresponding to these tableaux have some invariant subsequences under the action of the Knuth relations to illustrate their non-smoothness.

\subsection{Invariant subsequences with the pattern $3412$}
\begin{prop}
For $k \ge 3$, we have the fact that all elements of the following KL right cells are non-smooth:
\[
P=\scriptsize{\begin{tikzpicture}[scale=\domscale+0.3,baseline=-25pt]
\hobox{0}{0}{1}
\hobox{1}{0}{2}
\hobox{2}{0}{\dots}
\hobox{3}{0}{k}
\hobox{0}{1}{k+1}
\hobox{1}{1}{k+2}
\hobox{2}{1}{\dots}
\hobox{3}{1}{2k}
\end{tikzpicture}}
\]
and
\[
P={\scriptsize{ \begin{tikzpicture}[scale=\domscale+0.3,baseline=-26pt]
\hobox{0}{0}{1}
\hobox{1}{0}{2}
\hobox{2}{0}{\dots}
\hobox{3}{0}{k}
\hobox{4}{0}{k+1}
\hobox{0}{1}{k+2}
\hobox{1}{1}{k+3}
\hobox{2}{1}{\dots}
\hobox{3}{1}{2k+1}
\end{tikzpicture}}}.
\]

Actually, $(k+1,k+2,k-1,k)$ and $(k+2,k+3,k,k+1)$ are respectively the invariant subsequences.
\end{prop}

\begin{proof}
We only prove the first case and the second one is exactly the same. By Lemma \ref{ii1},  $k+1$ is always before $k$, $k+2$ is always after $k+1$ and $k-1$ is before $k$. If the relative position of $(k+1,k+2,k-1,k)$ will not change under the action of the Knuth relations,  we can prove the proposition. If not, the only possibility is that $k+2$ is after $k-1$ under some actions of the Knuth relations. But now there will exist an increasing subsequence $(1,2,\dots,k-1,k+2,\dots,2k)$ by Lemma \ref{ii1}. The length of the subsequence is $2k-2$. But we know that the longest increasing subsequence is the length of the first row by Proposition \ref{rowcolumn}. So we have $2k-2 \le k$, which means $k\leq 2$. This is a  contradiction since $k\geq 3$.

%

\end{proof}

For the following two cases, we can similarly prove that all elements of the given KL right cells are non-smooth:

\[
P=\scriptsize{\begin{tikzpicture}[scale=\domscale+0.3,baseline=-28pt]
\hobox{0}{0}{1}
\hobox{1}{0}{2}
\hobox{2}{0}{\dots}
\hobox{3}{0}{k}
\hobox{4}{0}{2k+1}
\hobox{0}{1}{k+1}
\hobox{1}{1}{k+2}
\hobox{2}{1}{\dots}
\hobox{3}{1}{2k}
\end{tikzpicture}}
\]
and
\[
P=\scriptsize{\begin{tikzpicture}[scale=\domscale+0.3,baseline=-35pt]
\hobox{0}{0}{1}
\hobox{1}{0}{2}
\hobox{2}{0}{\dots}
\hobox{3}{0}{k}
\hobox{0}{2}{2k+1}
\hobox{0}{1}{k+1}
\hobox{1}{1}{k+2}
\hobox{2}{1}{\dots}
\hobox{3}{1}{2k}
\end{tikzpicture}.}
\]

Now we give a sufficient condition  {implying} that all elements in a given KL right cell are non-smooth.

\begin{prop}
Let $P$ be a Young tableau and $l$ be the length of the first row. When $ P$ is  like:
\[
P=\scriptsize{\begin{tikzpicture}[scale=\domscale+0.3,baseline=-35pt]
\hobox{0}{0}{1}
\hobox{1}{0}{2}
\hobox{2}{0}{\dots}
\hobox{3}{0}{\dots}
\hobox{4}{0}{k}
\hobox{5}{0}{j}
\hobox{6}{0}{\dots}
\hobox{0}{1}{k+1}
\hobox{1}{1}{k+2}
\hobox{2}{1}{\dots}
\hobox{3}{1}{k+m}
\hobox{0}{2}{\vdots}
\end{tikzpicture}},
\]
and $k+m-2>l$,  { all elements of the KL right cell contain
the pattern $3412$,  thus in particular are non-smooth. }

\end{prop}

\begin{proof}
The subsequence $(k+1,k+2, k-1,k)$ in the row word of $P$ makes the pattern $3412$.   If $k+2$ is after $k-1$ under some actions of the Knuth relations, there will exist an increasing subsequence $(1,2,\dots,k-1,k+2,\dots,k+m)$ in the new permutation by Lemma \ref{ii1}. The length of this subsequence is $k+m-2$. But based on our assumption, it is larger than the length of the first row. This is a contradiction since Proposition \ref{rowcolumn}.
\end{proof}

Note that in the above proposition, we  have $k\geq m$.

\subsection{Invariant subsequences with the pattern $4231$}

We find that some special permutations will have some invariant subsequences satisfying the pattern $4231$.

\begin{prop}
If the Young tableau $P$  has only two columns and is  like the following:
\[
P=\scriptsize{\begin{tikzpicture}[scale=\domscale+0.3,baseline=-70pt]
\hobox{0}{0}{\vdots}
\hobox{1}{0}{\vdots}
\hobox{0}{1}{k-2}
\hobox{1}{1}{k-1}
\hobox{0}{2}{k}
\hobox{1}{2}{k+1}
\hobox{0}{3}{k+2}
\hobox{1}{3}{k+3}
\hobox{0}{4}{\vdots}
\hobox{1}{4}{\vdots}
\hobox{0}{5}{\vdots}
\end{tikzpicture}},
\]
then all elements of the right cell are non-smooth. Actually, the subsequence $(k+2,k,k+1,k-1)$ always exists  under the action of the Knuth relation, which satisfies the pattern $4231$.
\end{prop}

\begin{proof}
By Lemma \ref{ii1}, we find that $k+2$ is before $k+3$. If $k$ is before $k+2$, the subsequence $(k,k+2,k+3)$ exists, which implies the length of the first row of $P$ is larger than $2$. This is a contradiction.  If $k-1$ is before $k+1$, then the subsequence $(k-2,k-1,k+1)$ exists, which is also a contradiction. Thus, the subsequence $(k+2,k,k+1,k-1)$ always exists  under the actions of the Knuth relations.
\end{proof}
\begin{Rem}
 {Despite the preceding analysis that identifies right cells containing exclusively non-smooth elements through the detection of invariant subsequences matching the patterns $3412$ or $4231$, this characterization does not imply that \emph{every} element within a KL right cell with only non-smooth elements necessarily admits such invariant pattern-containing subsequences. Moreover, there exist examples of right cell with only non-smooth elements but there exist elements avoiding the pattern $3412$ and some elements avoiding $4231$. Here we give an example. Let $\pi_P=(4, 2, 6, 5, 3, 1)$. The corresponding right cell contains only non-smooth elements. However, Within this right cell, the element $(4, 2, 6, 5, 3, 1)$ contains $4231$ and avoids $3412$, while the element $(4, 6, 5, 2, 1, 3)$ contains $3412$ and  avoids $4231$.}

\end{Rem}

\section{Young tableaux with two columns}\label{twocolumn}

The  Young tableaux with only two columns   {frequently arise in} representation theory of Lie algebras and Lie groups, see for example \cite{BX,BXX}. Now we pay attention on KL right cells  corresponding to these special tableaux. Let the Young tableau $P$ be  as follows:

\[
P={\begin{tikzpicture}[scale=\domscale+0.1,baseline=-46pt]
\hobox{0}{0}{1}
\hobox{1}{0}{b}
\hobox{1}{1}{\vdots}
\hobox{0}{1}{\vdots}
\hobox{0}{2}{\vdots}
\hobox{0}{3}{\vdots}
\hobox{0}{4}{a}
\hobox{1}{2}{c}
\end{tikzpicture}.}
\]

We can easily find that the permutation $\pi_P=(a,\dots,1,c,\dots,b)$ is the column word corresponding to  $P$. Now we have the following proposition.

\begin{prop}
Let $P$ be a Young tableau with $\pi_P=(a,\dots,1,c,\dots,b)$ as above. If $P$ satisfies one of the following conditions:
\begin{enumerate}
	\item The numbers below $b$ are all larger than $a$.
	\item  There exist  numbers below $b$ smaller than $a$. Let $d$ be the largest one and  the elements of the first column which are smaller than $d$ are also smaller than $b$.
		
\end{enumerate}
then the KL right cell will have at least one smooth element.  {More precisely, the column word $(a, \dots, 1, c, \dots, b)$ is smooth if and only if (1) or (2) hold.}

\end{prop}

\begin{proof}
Firstly we know that column word $(a,\dots,1,c,\dots,b)$ avoids the pattern $3412$. So we focus on the pattern $4231$. Let us assume that $a$ is $``4"$ and $b$ is $``1"$. Now if the numbers below $b$ are larger than $a$, the pattern $4231$ will not appear. If there exist numbers below $b$ smaller than $a$, we denote the  largest one by $d$. If the first column has an element $e$ satisfying $b<e<d$, the pattern $4231$  will appear. Otherwise, the pattern $4231$ will not appear.

\end{proof}
For this point, we can further discuss the Young tableau of the following  {form}:

\[
P={\begin{tikzpicture}[scale=\domscale+0.1,baseline=-36pt]
\hobox{0}{0}{1}
\hobox{1}{0}{b}
\hobox{1}{1}{c}
\hobox{0}{1}{d}
\hobox{0}{2}{\vdots}
\hobox{0}{3}{a}
\end{tikzpicture}.}
\]

First, we have the fact that the corresponding KL right cells have non-smooth elements since Theorem \ref{main}. From the column word expression, we have the followings:
\begin{enumerate}[leftmargin=18pt]
\item If $c>a$  {which implies that $c=n$}, the KL right cell will have smooth elements.
\item If $c<a$ and the elements of the first column which are smaller than $c$ are also smaller than $b$,  {which indicates that $c=b+1$.} Then the KL right cell will have smooth elements.
\end{enumerate}

Now if the first column has elements which are smaller than $c$ and larger than $b$, we can solve one special case.

\begin{prop}\label{smoothnon}
Let the Young tableau $P$ be as follows:
\[
P={\begin{tikzpicture}[scale=\domscale+0.1,baseline=-45pt]
\hobox{0}{0}{1}
\hobox{1}{0}{2}
\hobox{1}{1}{c}
\hobox{0}{1}{3}
\hobox{0}{2}{b}
\hobox{0}{3}{\vdots}
\hobox{0}{4}{a}
\end{tikzpicture}.}
\]
Suppose $a>c$  and  $3<b<c$ in the first column. Then the KL right cell has both smooth and non-smooth elements.

\end{prop}

\begin{proof}
Let $b$ be the largest element in the first column which is smaller than $c$. Then $\pi_P=(a,\dots,b,\dots,3,c,1,2)$. Using the Knuth relations, we will have:
\begin{align*}
(a,\dots,b,\dots,3,c,1,2)&\stackrel{K}{\sim}(a,\dots,b,c,\dots,3,1,2 )\\
&\stackrel{K}{\sim}(b,a,\dots,c,\dots,3,1,2) \\
&\stackrel{K}{\sim}(b,a,\dots,c,\dots,1,3,2 )\\
&\stackrel{K}{\sim}\dots \\
&\stackrel{K}{\sim}(b,1,a,\dots,c,\dots,3,2).
\end{align*}

Obviously the permutation $(b,1,a,\dots,c,\dots,3,2)$ avoids the two patterns $3412$ and $4231$.
\end{proof}

\begin{Rem}
    For a given Young tableau $P$, the KL right cell $\mathcal{C}_R(P)$ will contain smooth elements if the column word $\pi_P^{c}$ or row word $\pi_P$ avoids the two patterns $3412$ and $4231$.  { For example, the  element $w=(4,3,2,1,7,6,5,9,8)$  is smooth since it  avoids the two patterns $3412$ and $4231$. Note that $w$ is the column word of the Young tableau $P(w)$. Thus the KL right cell $\mathcal{C}_R(P(w))$ contains smooth elements.}

 { But it can happen that both $\pi_P^{c}$ and   $\pi_P$ are not smooth when $\mathcal{C}_R(P)$  contains smooth elements.}
    For \[
P={\begin{tikzpicture}[scale=\domscale+0.1,baseline=-36pt]
\hobox{0}{0}{1}
\hobox{1}{0}{2}
\hobox{1}{1}{5}
\hobox{0}{1}{3}
\hobox{0}{2}{4}
\hobox{0}{3}{6}
\end{tikzpicture}},
\]
its column word $\pi_P^{c}=(6,4,3,1,5,2)$ contains the pattern $4231$ (since we can choose the subsequence $(6,3,5,2)$) and its row word $\pi_P=(6,4,3,5,1,2)$ contains the pattern $3412$ (since we can choose the subsequence $(3,5,1,2)$). But from Proposition \ref{smoothnon}, we know that the KL right cell $\mathcal{C}_R(P)$ contains both smooth and non-smooth elements. For example $w=(4,1,6,5,3,2)\in \mathcal{C}_R(P)$ is smooth.
\end{Rem}

\section{Right cells containing some smooth elements}\label{algorithm}
In this section, we want to give some algorithms to determine that a right cell contains a smooth element or not by using the Knuth relations.

 To give our algorithm, we recall the famous hook formula, which was found by Frame, Robinson and Thrall \cite{FRT}.

\begin{definition}
    If $v=(i,j)$ is a node in the diagram or Young tableau $P$, then it has hook
    $$H_v=H_{i,j}=\{(i,j')\mid j'\geq j\}\cup \{(i',j)\mid i'\geq i\}$$
    with corresponding hooklength
    $$h_v=h_{i,j}=|H_{i,j}|.$$
\end{definition}

\begin{example}
   Let \[
P={\begin{tikzpicture}[scale=\domscale+0.1,baseline=-36pt]
\hobox{0}{0}{1}
\hobox{1}{0}{2}
\hobox{1}{1}{5}
\hobox{0}{1}{3}
\hobox{0}{2}{4}
\hobox{0}{3}{6}
\end{tikzpicture}},
\]
then the dotted cells in
\[
{\begin{tikzpicture}[scale=\domscale+0.1,baseline=-30pt]
\hobox{0}{0}{}
\hobox{1}{0}{}
\hobox{1}{1}{\bullet}
\hobox{0}{1}{\bullet}
\hobox{0}{2}{\bullet}
\hobox{0}{3}{\bullet}
\end{tikzpicture}}
\]
are  the hook $H_{2,1}$ with hooklength
$h_{2,1}=4$.
\end{example}

 {Based on Proposition \ref{prop::2.11} and \ref{prop::2.12}, we can rephrase the standard hook formula and determinantal formula as follows.}

\begin{prop}[Hook formula \cite{FRT} ]\label{hook}
Let $P$ be a standard Young tableau and $\mathcal{C}$ be the corresponding Kazhdan--Lusztig right cell in the symmetric group $S_n$. Then
$$\#\mathcal{C}=\frac{n!}{\prod\limits_{(i,j)\in P} h_{i,j}}.$$

\end{prop}

There is another formula which is much older than the hook formula. In  the following, we set $1/r!=0$ if $r<0$ and $0!=1$.
\begin{prop}[Determinantal formula \cite{Sagan} ]\label{deter}
Let $P$ be a standard Young tableau with shape $p=[p_1,...,p_l]$ and $\mathcal{C}$ be the corresponding Kazhdan--Lusztig right cell in the symmetric group $S_n$. Then
$$\#\mathcal{C}=\frac{n!}{|(a_{ij})_{l\times l}|},$$
where $(a_{ij})_{l\times l}$ is a $l$ by $l$ matrix with $a_{ij}=\frac{1}{(p_i-i+j)!}$.
\end{prop}
Some more details about this formula can be found in Sagan \cite[\S 3.11]{Sagan}.

Now recall that we have $x \stackrel{K}{\sim}y$ if and only if $P(x)=P(y)$ for any two elements $x$ and $ y$ in $S_n$.
Let $w\in S_n$ and $P(w)$ be the corresponding Young tableau which corresponds to a right cell $\mathcal{C}_R(P(w))$.
Then we have the following algorithm to count the number of smooth elements in the right cell $\mathcal{C}_R(P(w))$:
\begin{enumerate}[leftmargin=18pt]
    \item If $w$ is smooth, then we put it in a set and denote it by $\mathcal{S}_1$. If $w$ is not smooth, we write $w=(w_1,w_2,w_3,...,w_n)$ and put it in another set and denote it by $\mathcal{N}_1$. Thus we have $\#\mathcal{S}_1=1$ or $\#\mathcal{N}_1=1$;
    \item From $w=(w_1,w_2,w_3,...,w_n)$, we consider any triple $(w_i,w_{i+1},w_{i+2})$ for $1\leq i\leq n-2$. If it is neither a decreasing nor an increasing sequence, we can use the Knuth relation to get a different element $w'_{\bar{i}}$ such that $w'_{\bar{i}}\stackrel{K}{\sim} w $. Let  $G_1:=\{w'_{\bar{i}}\mid 1\leq i\leq n-2\}$. Then we extract all smooth elements from $G_1$ and denote this new set by  $\mathcal{S}'_1$. Denote $\mathcal{N}'_1=G_1- \mathcal{S}'_1$. Then we define $\mathcal{S}_2=\mathcal{S}'_1\cup \mathcal{S}_1$ and $\mathcal{N}_2=\mathcal{N}'_1\cup \mathcal{N}_1$;
    \item  For each element in the set $G_1$, we repeat the procedure described in step $2$. This will yield new sets of smooth and non-smooth elements, denoted as and get some new smooth elements and non-smooth elements, denoted as $\mathcal{S}'_2$ and $\mathcal{N}'_2$ respectively.  Then we define $\mathcal{S}_3=\mathcal{S}'_2\cup \mathcal{S}_2$ and $\mathcal{N}_3=\mathcal{N}'_2\cup \mathcal{N}_2$;
    \item We continue the above process. Since the right cell $\mathcal{C}_R(P(w))$ has a finite number of elements, and every element in $\mathcal{C}_R(P(w))$ is Knuth equivalent to $w$,  the process will terminate after a finite number of steps. We denote the set of all smooth elements obtained throughout this process as $\mathcal{S}$, and the set of all non-smooth elements as  $\mathcal{N}$.
    Then we have $\#\mathcal{S}+\#\mathcal{N}=\#\mathcal{C}_R(P(w))$.

\end{enumerate}

The above algorithm is called \textit{smooth elements finding algorithm} (SEF algorithm for short).

\begin{example}
 Let \[
P={\begin{tikzpicture}[scale=\domscale+0.1,baseline=-36pt]
\hobox{0}{0}{1}
\hobox{1}{0}{2}
\hobox{1}{1}{5}
\hobox{0}{1}{3}
\hobox{0}{2}{4}
\hobox{0}{3}{6}
\end{tikzpicture}},
\]
then the hooklengths are: $h_{11}=5$, $h_{1,2}=2,$
$h_{2,1}=4$, $h_{2,2}=1$, $h_{3,1}=2$, $h_{4,1}=1$. Thus we have $\#\mathcal{C}_R(P(w))=\frac{6!}{5\cdot 2\cdot 4\cdot 1\cdot 2\cdot 1}=9.$
Let $w=(6,4,3,1,5,2)$ be the column word of $P$. Then $w$ contains the pattern $4231$ (since we can choose the subsequence $(6,3,5,2)$) and is not smooth. We consider the triples $(w_i,w_{i+1},w_{i+2})$ and find that there are only two triples: $(3,1,5)$ and $(1,5,2)$ from which we can use the Knuth relation to get new elements. Thus we get $w'_{\bar{3}}=(6,4,3,5,1,2)=w'_{\bar{4}}$, which is not smooth. Then from $w'_{\bar{3}}$, we can use the triples: $(4,3,5)$, $(3,5,1)$ and $(5,1,2)$. Thus we get a new element ${(w'_{\bar{3}})}_{\bar{2}}=(6,4,5,3,1,2)$. Here ${(w'_{\bar{3}})}_{\bar{3}}={(w'_{\bar{3}})}_{\bar{4}}=(6,4,3,1,5,2)=w$ is not a new element.
Then from ${(w'_{\bar{3}})}_{\bar{2}}$, we can use the triples: $(6,4,5)$, $(4,5,3)$ and $(3,1,2)$. Thus we get new elements $({(w'_{\bar{3}})}_{\bar{2}})_{\bar{1}}=(4,6,5,3,1,2)$ and $({(w'_{\bar{3}})}_{\bar{2}})_{\bar{4}}=(4,6,5,1,3,2)$. Here $({(w'_{\bar{3}})}_{\bar{2}})_{\bar{2}}=(6,4,3,5,1,2)=w'_{\bar{3}}$ is not a new element. We continue and get new elements
$(({(w'_{\bar{3}})}_{\bar{2}})_{\bar{1}})_{\bar{4}}=(4,6,5,1,3,2)$ and $(({(w'_{\bar{3}})}_{\bar{2}})_{\bar{4}})_{\bar{3}}=(6,4,1,5,3,2)$. Then we continue and get $((({(w'_{\bar{3}})}_{\bar{2}})_{\bar{1}})_{\bar{4}})_{\bar{3}}=(4,6,1,5,3,2)$ and $(((({(w'_{\bar{3}})}_{\bar{2}})_{\bar{1}})_{\bar{4}})_{\bar{3}})_{\bar{1}}=(4,1,6,5,3,2)$. We find that there is only one smooth element among these $9$ elements, i.e., $(4,1,6,5,3,2)$.
We draw the process as follows:









$$\begin{tikzpicture}[semithick,->]
 \Tree
[.(6,4,3,1,5,2)
[.(6,4,3,5,1,2) [.(6,4,5,3,1,2) [.(4,6,5,3,1,2) [.(4,6,5,1,3,2) [.(4,6,1,5,3,2) [.(4,1,6,5,3,2) ] ] ] ] [.(6,4,5,1,3,2) [.(6,4,1,5,3,2) ] ] ] ]
]
\end{tikzpicture}$$

If we use our program   ``Young", the input is $\{6~ 4~ 3~ 1~ 5~ 2\}$. The output will be the following:
\begin{verbatim}
P tableau:
1 2
3 5
4
6

Q tableau:
1 5
2 6
3
4

Hook lengths:
5 2
4 1
2
1

Number of standard Young tableaux for shape (2, 2, 1, 1): 9
Number of smooth permutations: 1
Number of non-smooth permutations: 8
Number of smooth permutations: {(4, 1, 6, 5, 3, 2)}
\end{verbatim}

Thus we have the same result but it is much faster to get these results.

\end{example}



\begin{Cor}\label{findsmooth}
    Let $L_w$ be a highest weight module of $\mathfrak{sl}(n,\mathbb{C})$. If we can find a smooth element in the KL right cell $\mathcal{C}_R(P(w))$ by using our SEF algorithm, we will have
 $ V(L_w)=\mathcal{V}(w)$.
\end{Cor}
\begin{proof}
    From Proposition \ref{constant}, we know that  $V(L_w)$ is a constant in the KL right cell $\mathcal{C}_R(P(w))$. If we can find  a smooth element $x$ in the KL right cell $\mathcal{C}_R(P(w))$ by using our SEF algorithm,  we will have $ V(L_x)=\mathcal{V}(x)$ by Proposition \ref{smoothequal}. Thus by Proposition \ref{constant}, we will have $ V(L_w)=V(L_x)=\mathcal{V}(x)$. By Proposition \ref{geometriccell}, $\mathcal{V}(x)=\mathcal{V}(w)$ since $x \stackrel{R}{\sim}w$. Therefore, we have $$ V(L_w)=\mathcal{V}(x)=\mathcal{V}(w).$$
\end{proof}

 {We also investigate the proportion of right cells containing a smooth element, defined as the number of right cells with smooth elements divided by the total number of right cells. Our results indicate that this proportion decreases as $n$ increases, showing that right cells containing smooth elements become increasingly rare.}
\begin{figure}[H]
    \centering
\includegraphics[width=12cm]{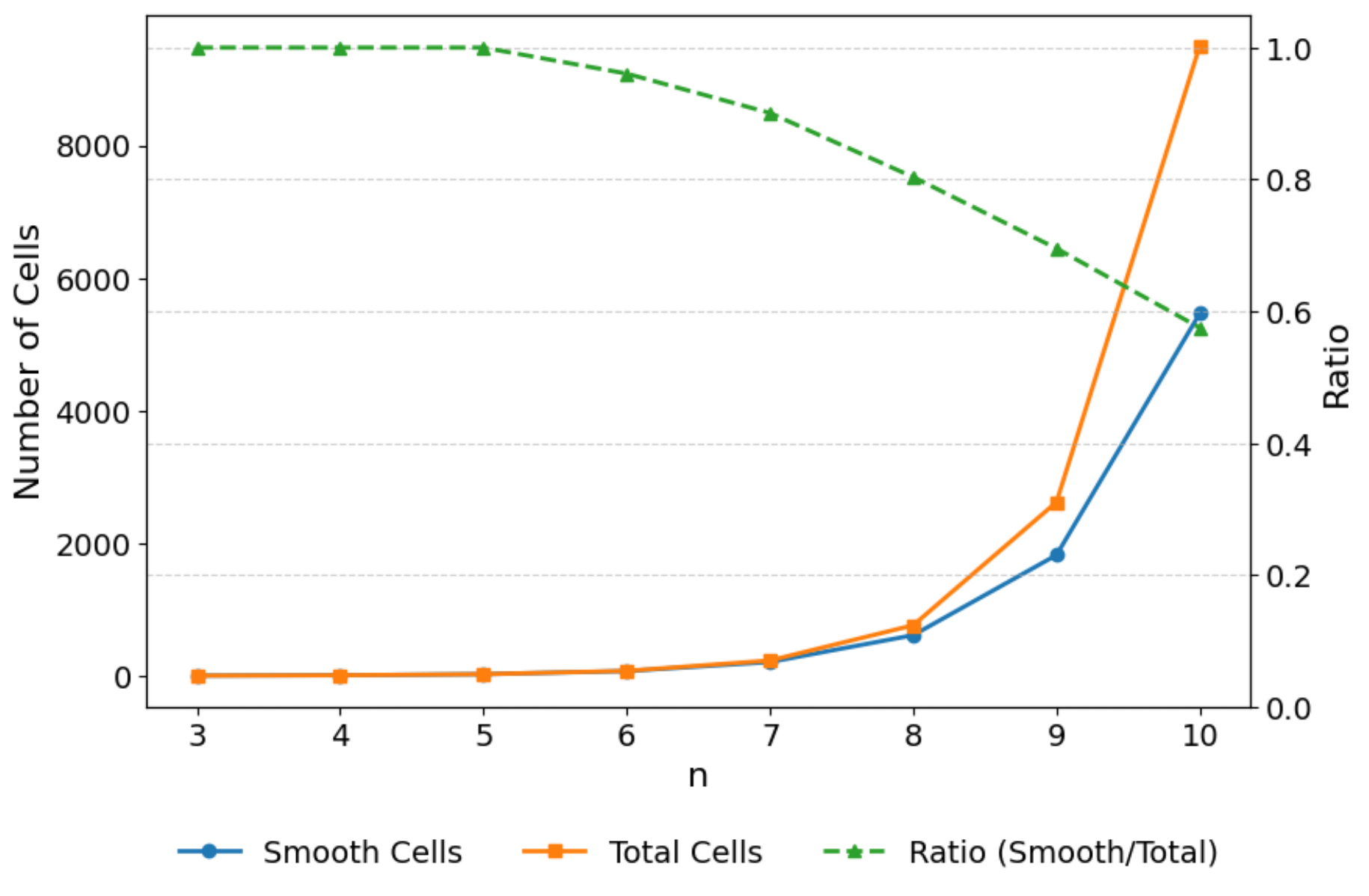}
    \caption{Cells containing smooth elements, total cells and their ratio}
    \label{fig:placeholder}
\end{figure}

\subsection*{Acknowledgments}
Z. Bai is
supported by NSFC Grant No. 12171344. We would like to thank the referee for very helpful suggestions and comments.

\end{document}